\documentclass[11pt,a4paper]{amsart}
\usepackage[utf8]{inputenc}
\usepackage[T1]{fontenc}
\usepackage{lmodern}
\usepackage{amsmath,amssymb,amsfonts,amsthm, bbm, bm}
\usepackage[margin=1.0in]{geometry}
\usepackage{graphicx}
\usepackage{caption}
\usepackage{enumitem}
\usepackage{url}
\usepackage{lipsum}
\usepackage{hyperref}
\usepackage{nameref}
\usepackage{xcolor}
\usepackage[english]{babel}

\newtheorem{theorem}{Theorem}
\newtheorem{lemma}{Lemma}
\newtheorem{remark}{Remark}
\newtheorem{corollary}{Corollary}

\theoremstyle{definition}
\newtheorem*{assumption}{Assumption}

\newcommand{\norm}[1]{\|#1\|}

\newcommand{\calP}{\mathcal{P}}

\newcommand{\E}{\mathbb{E}}

\renewcommand{\P}{\mathbb{P}}
\newcommand{\R}{\mathbb{R}}

\begin{document}

\title[Uniqueness for finite-state MFG with non-separable Hamiltonian]{A uniqueness result for finite-state mean field games with non-separable Hamiltonian} 

\thanks{
	A.C. and N.F. are supported by the project MeCoGa ``Mean field control and games'' of the University of Padova through the program STARS@UNIPD. A.C. also acknowledges support  from 
	the PRIN 2022 Project 2022BEMMLZ ``Stochastic control and games and the role of information'', 
	the INdAM-GNAMPA Project 2025 ``Stochastic control and MFG under asymmetric information: methods and applications'' and the PRIN 2022 PNRR Project P20224TM7Z  ``Probabilistic methods for energy transition''.
}




\author{%
	Alekos Cecchin}

\author{Luca Di Persio}

\author{Nicola Fraccarolo}

\address[A. Cecchin and N. Fraccarolo]{Department of Mathematics ``T. Levi-Civita'', University of Padova, 
	\newline \indent Via Trieste 63, 35121 Padova, Italy} 
\email{alekos.cecchin@unipd.it, nicola.fraccarolo@unipd.it}

\address[L. Di Persio]{
	Department of Computer Science, University of Verona, 
	\newline \indent Strada le Grazie 15, 37134 Verona, Italy}      
\email{luca.dipersio@univr.it}

	
	
	\date{April 18, 2026}

	\begin{abstract}
		We study a class of continuous-time mean field games on a finite state space with transition rates depending on the population distribution, leading to a non-separable Hamiltonian. In this setting, classical Lasry--Lions monotonicity arguments do not apply directly.
		We establish a new uniqueness result on arbitrary time horizons under a combination of strong monotonicity assumptions on the costs and quantitative bounds on the interaction term in the dynamics. 
	
	\end{abstract}

\maketitle

\section{Introduction}

Mean field games (MFGs), introduced independently by Lasry and Lions \cite{LasryLions} and by Huang, Malham\'e and Caines \cite{Huang2006}, provide a robust framework for modelling strategic interactions among a large population of rational agents. In the asymptotic regime where the number of players tends to infinity, the interaction is captured through the distribution of the state of a representative agent, leading to a coupled system of a Hamilton-Jacobi-Bellman (HJB) equation and a Kolmogorov forward equation. We refer to the monographs \cite{CardaliaguetDelarueLasryLions, CarmonaDelarue_book_I, CarmonaDelarue_book_II} for a comprehensive introduction to the theory and its probabilistic and analytical foundations.

The question of existence of solutions for MFG systems is by now rather well understood. In a nutshell, solutions to the mean field game are given by a fixed point and thus general existence results have been established under broad assumptions using both PDE and probabilistic techniques; see for instance \cite{CarmonaDelarue_book_I, lacker2015mean}. In the finite-state setting, existence can be obtained under relatively mild conditions, see in particular \cite{CecchinFischer}. 

A central issue in the theory is the uniqueness of solutions. From a modelling perspective, uniqueness is crucial as it ensures the predictability of equilibrium behaviour. From a numerical viewpoint, it ensures the convergence of approximation schemes. However, uniqueness may fail even in simple models, and counterexamples are well documented in the literature; see for instance \cite{bardi2019non, delarue2020selection} and, in the two-state setting, \cite{cecchin2019convergence, dai2019climb}. 

Uniqueness results can be obtained under short time horizon assumptions, as the fixed point becomes a contraction; see \cite{achdou2021mean, bardi2019non},  or \cite{cardaliaguet2022splitting} via splitting techniques.
A classical approach to uniqueness for any time horizon relies on monotonicity conditions. This is well understood in the case when the Hamiltonian of the problem splits as the sum of two terms, one depending on the momentum variable and another depending on the distribution of the population: this is called the separable case.  
In this case, the Lasry-Lions monotonicity condition ensures uniqueness of equilibria \cite{LasryLions, CardaliaguetDelarueLasryLions, mou2024wellposedness}, and has been extended in several directions, including displacement monotonicity \cite{gangbo2022global, meszaros2024mean} and other generalised monotonicity frameworks \cite{graber2023monotonicity, mou2024mean, mou2025second}. Analogous results are available in the finite-state setting; see for instance \cite{bertucci2019some, gomes2013continuous}.

However, in several applications the dynamics of the representative player depend explicitly on the population distribution, as in models with congestion effects, network interactions, or endogenous jump intensities; see e.g. \cite{ghattassi2025, LauriereSongTang2023}. In these cases the Hamiltonian becomes non-separable, and the standard Lasry-Lions monotonicity arguments no longer apply. Some uniqueness results in this direction have been obtained, in case of diffusion-based models; see \cite[Theorem 1.13]{achdou2021mean} following the lectures \cite{Lionscollege2}, and \cite{gangbo2022mean}. Nevertheless, these conditions are often difficult to verify when the interaction enters directly through the dynamics via distribution-dependent transition rates.

In this paper we investigate the uniqueness of mean field game systems with non-separable Hamiltonians induced by distribution-dependent dynamics. We focus on finite-state continuous-time models, which provide a tractable yet expressive framework allowing for sharp analytical results. Within this setting, we establish sufficient conditions ensuring uniqueness of solutions over arbitrary time horizons. Our approach highlights the precise role played by the interaction through the transition rates and shows how classical monotonicity conditions must be quantitatively reinforced to compensate for the lack of separability.

\subsection*{Outline of the paper}

The paper is organised as follows. 
In Section~\ref{sec:model}, we introduce the finite-state mean field game model, define the controlled dynamics and the associated cost functional, and derive the Hamiltonian and the MFG system. 
The set of assumptions is stated in Section~\ref{sec:assumptions}.
Section~\ref{sec:uniqueness} is devoted to the proof of the uniqueness result. 
In Section~\ref{sec:gradient estimate}, we discuss
uniform bounds on the finite differences of the value function.
Finally, Section~\ref{sec:conclusion} concludes the article.

\section{Main Result}

\subsection{Model Setup}\label{sec:model}

We consider the continuous time evolution of the state $X(t)$ of a representative player belonging to a finite set $\Sigma = \{1,\ldots,d\}$ over the interval $[0,T]$.
The player only knows its position and the distribution of the system, which is an element of the simplex of probability measures on $\Sigma$
\begin{equation*}
    \calP(\Sigma) = \Big\{ \mu = (\mu(1),\ldots,\mu(d)) \in [0,1]^{\Sigma} \colon \sum_{x\in\Sigma}\mu(x) = 1 \Big\}.
\end{equation*}
For a difference of vectors $\mu, \tilde\mu\in\calP(\Sigma)$, we denote $||\mu-\tilde\mu||_1= \sum_{x\in \Sigma} |\mu(x) - \tilde\mu (x)|$.

The player is allowed to control their transition rate via feedback controls $\alpha \colon [0,T] \times \Sigma \to [0,+\infty)^{\Sigma}$, where $\alpha_y(t,x)$ denotes the $y$-th coordinate of $\alpha(t,x)$.
Given two states $x$ and $y$, with $x \neq y$, the rate at which the player jumps from $x$ to $y$ is
\begin{align*}
    \P(X(t+h) = y \,|\, X(t) = x)
    = [\alpha_{y}(t,x) + b(x,\mu(t))]h + o(h),
\end{align*}
where $b \colon \Sigma \times \calP(\Sigma) \to [0,+\infty)$ is a given function, and $\mu \colon [0,T] \to \calP(\Sigma)$ is a deterministic flow of probability measures, that is, 
$\mu(t)= (\mu(t,1), \dots, \mu(t,d))$.
The control $\alpha_y(t,x)$ represents the endogenous part of the transition intensity from state $x$ to state $y$. The function $b(x,\mu)$ models an exogenous interaction term that accounts for the dependence of the dynamics on the population distribution. This interaction is the source of the Hamiltonian's non-separability.
Given the functions
\begin{align*}
    &L \colon \Sigma \times [0,+\infty)^{\Sigma} \to \R, 
&&L(x,a) = \frac{1}{2} \sum_{y \neq x} a_{y}^2, 
    \\
    &f \colon \Sigma \times \calP(\Sigma) \to \R,
    &&g \colon \Sigma \times \calP(\Sigma) \to \R, 
\end{align*}
for a fixed $\mu$, we let $J^\mu \colon [0,T] \times \Sigma \times \calP(\Sigma) \to \R$ be defined by
\begin{align*}
    J^\mu(t,x,\alpha) = \E\left[\int_{t}^{T}[L(X(s),\alpha(s,X(s))) + f(X(s),\mu(s))]ds + g(X(T),\mu(T)) \,\Big|\, X(t)=x \right].
\end{align*}
We define the value function associated to $\mu$, denoted by $u^\mu \colon [0,T] \times \Sigma \to \R$, as
\begin{equation}
    u^\mu(t,x) = \inf_{\alpha} J^\mu(t,x,\alpha).
\end{equation}

Given a function $g : \Sigma \to \R$, we denote its first finite difference by
$\Delta^yg(x)= g(y)-g(x)$, for $x,y\in\Sigma$, and denote the vector $\Delta g(x) \in \R^\Sigma$ by
\begin{equation*}
    \Delta g(x) = \big( g(y)-g(x) \big)_{y\in\Sigma}.
\end{equation*}

The Hamiltonian $H \colon \Sigma \times \calP(\Sigma) \times \R^{\Sigma} \to \R$ is given by
\begin{align*}
    H(x,\mu,u)
    &= \sup_{a\in[0,+\infty)^\Sigma} \Bigg\{-\sum_{y \neq x}[a_{y} + b(x,\mu)] [u(y) - u(x)] - \frac{1}{2}\sum_{y \neq x} a_{y}^2\Bigg\} \\
    &= \sum_{y \neq x} \sup_{a\in[0,+\infty)} \Big\{-[a + b(x,\mu)][u(y) - u(x)] - \frac{1}{2} a^2\Big\} \\
    &= \sum_{y \neq x} \Bigg[\sup_{a\in[0,+\infty)} \Big\{-a[u(y) - u(x)] - \frac{1}{2}a^2\Big\} - b(x,\mu)[u(y) - u(x)]\Bigg] \\
    &= \sum_{y \neq x} \Bigg[\frac{1}{2}[u(y) - u(x)]_{-}^2 - b(x,\mu)[u(y) - u(x)]\Bigg] \\
    &= \sum_{y \neq x} \Bigg[\frac{1}{2}[\Delta^y u(x)]_{-}^2 - b(x,\mu) {\Delta^y}u(x)\Bigg].
\end{align*}
The supremum is reached at $a_y^*(x,\Delta u) = [\Delta^{y}u(x)]_{-}$ and gives rise to the following optimal jump rate:
\begin{align*}
    \lambda_y^*(x,\Delta u) = a_y^*(x,\mu,\Delta u) + b(x,\mu) = [\Delta^{y}u(x)]_{-} + b(x,\mu).
\end{align*}
Since $u$ enters the definition only through the vector $\Delta u(x)$, we can rewrite the Hamiltonian as
\begin{align*}
    H(x,\mu,p) = \sum_{y \neq x} \left[\frac{1}{2}(p_y)_{-}^{2} - b(x,\mu)p_y\right],
\end{align*}
where 
\[ 
p_-^2 \equiv \bigl(p_-\bigr)^2,\qquad p_-:=\max\{-p,0\}. 
\]
The optimal control problem associated with a fixed flow of measures $\mu$ leads to a system of coupled equations. The HJB equation characterises the value function of the representative agent, while the forward equation describes the evolution of the population distribution under the optimal feedback control.
The general mean field game system is
\begin{equation}\label{eqn:MFG system}
    \begin{cases}
        -\partial_t u(t,x) 
        + \displaystyle\sum_{y \neq x} 
        \bigg[
        \frac{1}{2} [\Delta^y u(t,x)]^2_- 
        - b(x,\mu(t))\Delta^y u(t,x) 
        \bigg]
        - f(x,\mu(t)) 
        = 0, \\
        \partial_t \mu(t,x) 
        = \displaystyle\sum_{y \neq x} 
        \Big[
        \mu(t,y) \big( [\Delta^{x}u(t,y)]_{-} + b(y,\mu(t)) \big)
        - \mu(t,x) \left( [\Delta^{y}u(t,x)]_{-} + b(x,\mu(t)) \right)
        \Big],\\
        u(T,x) = g(x,\mu(T)), \\
        \mu(0) = \mu_{0}.
    \end{cases}
\end{equation}

Let $\norm{\Delta u}_\infty = \sup_{t \in [0,T]} \max_{x,y \in \Sigma} |\Delta^y u(t,x)|$.

\subsection{Assumptions}\label{sec:assumptions}

We introduce a set of assumptions ensuring uniqueness of the equilibrium, combining monotonicity and regularity properties. In particular, we require suitable monotonicity conditions on the running and terminal costs, together with Lipschitz continuity and nonnegativity of the interaction term entering the dynamics.
\begin{assumption}\label{ass:uniqueness}
    ~
    \begin{enumerate}
        \item[(A1)] There exists a constant $C_f \ge 0$ such that
        \begin{equation*}
            \sum_{x \in \Sigma} [f(x,\mu) - f(x,\tilde{\mu})][\mu(x) - \tilde{\mu}(x)] \ge C_f \norm{\mu - \tilde{\mu}}_1^2, \qquad \forall \mu, \tilde{\mu} \in \calP(\Sigma).
        \end{equation*}
        \item[(A2)] $g$ is monotone: 
        \begin{equation*}
            \sum_{x \in \Sigma} [g(x,\mu) -g(x,\tilde{\mu})][\mu(x) - \tilde{\mu}(x)] \geq 0 \qquad \forall \mu, \tilde{\mu} \in \calP(\Sigma).
        \end{equation*}
        \item[(A3)] For every $x \in \Sigma$, the function $b(x,\cdot)$ is nonnegative and such that:
        \begin{equation*}
            |b(x,\mu) - b(x,\tilde{\mu})| \leq L_{b} \left[\prod_{z \neq x} \mu(z) + \prod_{z \neq x} \tilde{\mu}(z) \right] \norm{\mu - \tilde{\mu}}_{1} \qquad \forall \mu, \tilde{\mu} \in \calP(\Sigma).
        \end{equation*}
        \item[(A4)] The constants $L_b$, $C_f$ and $d$ are such that
        \begin{align*}
            &C_f > d(d-1) [9 L_b^2 + 2 \norm{\Delta u}_\infty L_b].
        \end{align*}
    \end{enumerate}
\end{assumption}

\begin{remark}
    An example of function $b$ that satisfies Assumption~(A3) is given by
    \begin{equation*}
        b(x,\mu) = \bigg( \prod_{y \neq x} \mu(y) \bigg)^2.
    \end{equation*}
    Note that (A4) depends on the discrete gradient $||\Delta u||$ of the solution. We provide in Subsection~\ref{sec:gradient estimate} estimates of $||\Delta u||_\infty$ in terms of the data of the problem. 
\end{remark}

\subsection{Uniqueness Result}\label{sec:uniqueness}

We now state our uniqueness result.
\begin{theorem}
    Under Assumption~\ref{ass:uniqueness}, the solution $(u,\mu)$ of \eqref{eqn:MFG system} is unique.
\end{theorem}
\begin{proof}
Let $(u,\mu)$ and $(\tilde{u},\tilde{\mu})$ be solutions to \eqref{eqn:MFG system}.
Define the function 
$$\varphi(t) := [u(t)-\tilde{u}(t)] \cdot [\mu(t)-\tilde{\mu}(t)]
= \sum_{x \in \Sigma} [u(t,x)-\tilde{u}(t,x)][\mu(t,x)-\tilde{\mu}(t,x)]$$ 
for $t \in [0,T]$.

\emph{Step 1: the function $\varphi$ is monotone decreasing for $t \in [0,T]$.}\\
We need to show that
$
\frac{d}{dt} \varphi(t) \le 0,
$
that is,
\begin{equation*}
    \frac{d}{dt} \sum_{x \in \Sigma} [u(t,x)-\tilde{u}(t,x)][\mu(t,x)-\tilde{\mu}(t,x)] \le 0.
\end{equation*}
From now on, we will omit the time variable $t$.
To compute the derivative, we need to compute the two terms $(I)$ and $(II)$:
\begin{align*}
    &\frac{d}{dt} \sum_{x\in\Sigma} [u(x)-\tilde{u}(x)]\big[\mu(x)-\tilde{\mu}(x)\big] \\
    &= \sum_{x} \left[\frac{d}{dt}[u(x)-\tilde{u}(x)]\right]\big[\mu(x)-\tilde{\mu}(x)\big] &&(I)\\
    &\qquad + \sum_{x} [u(x)-\tilde{u}(x)]\left[\frac{d}{dt}\big[\mu(x)-\tilde{\mu}(x)\big]\right] &&(II)
\end{align*}
Developing the two terms, we have
\allowdisplaybreaks
\begin{align*}
    (I)
    &= \sum_{x} 
    \Bigg[
    \sum_{y\neq x}
    \left( \frac{1}{2}[\Delta^{y}u(x)]_{-}^2 - b(x,\mu)\Delta^{y}u(x) \right)
    - f(x,\mu) \\
    &\qquad\quad - \sum_{y\neq x} 
    \left( \frac{1}{2}[\Delta^{y}\tilde{u}(x)]_{-}^2 - b(x,\tilde{\mu})\Delta\tilde{u}(x) \right) 
    + f(x,\tilde{\mu}) 
    \Bigg]
    \big[\mu(x)-\tilde{\mu}(x)\big] \\
    &= \sum_{x}\sum_{y \neq x}
    \Bigg[
    \frac{1}{2}[\Delta^{y}u(x)]_{-}^2 - b(x,\mu)\Delta^{y}u(x)
    - \frac{1}{2}[\Delta^{y}\tilde{u}(x)]_{-}^2 + b(x,\tilde{\mu})\Delta^{y}\tilde{u}(x)
    \Bigg] 
    \big[\mu(x)-\tilde{\mu}(x)\big] \\
    &\quad + \sum_{x} \big[-f(x,\mu)+f(x,\tilde{\mu})\big] \big[\mu(x)-\tilde{\mu}(x)\big], \\
    %
    %
    %
    %
    (II)
    &= \sum_{x} [u(x)-\tilde{u}(x)] 
    \Bigg[
    \sum_{y \neq x} \Big( \mu(y)\big([\Delta^{x}u(y)]_{-} + b(y,\mu)\big) 
    - \mu(x)\big([\Delta^{y}u(x)]_{-} + b(x,\mu)\big) \Big) \\
    &\qquad\qquad\qquad\qquad\quad - \sum_{y\neq x} \Big( \tilde{\mu}(y)\big([\Delta^{x}\tilde{u}(y)]_{-} + b(y,\tilde{\mu})\big)
    + \tilde{\mu}(x)\big([\Delta^{y}\tilde{u}(x)]_{-} + b(x,\tilde{\mu})\big) \Big)
    \Bigg] \\
    &= \sum_{x}\sum_{y\neq x} 
    \Bigg(
    \mu(x) \Big[ [\Delta^{y}u(x)]_{-} + b(x,\mu) \Big] \big[\Delta^{y}u(x)-\Delta^{y}\tilde{u}(x)\big] \\
    &\qquad\qquad\qquad -\tilde{\mu}(x) \Big[ [\Delta^{y}\tilde{u}(x)]_{-} + b(x,\tilde{\mu}) \Big] \big[\Delta^{y}u(x)-\Delta^{y}\tilde{u}(x)\big]
    \Bigg) \\
    &= \sum_{x}\sum_{y\neq x} \big[\Delta^{y}u(x) - \Delta^{y}\tilde{u}(x)\big]
    \Big[
    \mu(x)\big[[\Delta^{y}u(x)]_{-}+b(x,\mu)\big] 
    - \tilde{\mu}(x)\big[[\Delta^{y}\tilde{u}(x)]_{-}+b(x,\tilde{\mu})\big]
    \Big]
\end{align*}

Putting it all together, we have
\begin{align*}
    \frac{d}{dt} \varphi(t)
    &= \sum_{x}\sum_{y\neq x}
    \left[
    \frac{1}{2}[\Delta^{y}u(x)]_{-}^{2} - b(x,\mu)\Delta^{y}u(x) 
    - \frac{1}{2}[\Delta^{y}\tilde{u}(x)]_{-}^{2} + b(x,\tilde{\mu})\Delta^{y}\tilde{u}(x)
    \right]
    \big[\mu(x)-\tilde{\mu}(x)\big] \\
    &\quad + \sum_{x} \big[-f(x,\mu)+f(x,\tilde\mu)\big] \big[\mu(x)-\tilde{\mu}(x)\big] \\
    &\quad + \sum_{x}\sum_{y\neq x} \big[\Delta^{y}u(x) - \Delta^{y}\tilde{u}(x)\big]
    \Big[
    \mu(x)\big[[\Delta^{y}u(x)]_{-}+b(x,\mu)\big] 
     - \tilde{\mu}(x)\big[[\Delta^{y}\tilde{u}(x)]_{-}+b(x,\tilde{\mu})\big]
    \Big] \\
    %
    &= \sum_{x}\sum_{y\neq x}
    \left[
    \frac{1}{2}[\Delta^{y}u(x)]_{-}^{2} - \frac{1}{2}[\Delta^{y}\tilde{u}(x)]_{-}^{2}
    \right]
    \big[\mu(x)-\tilde{\mu}(x)\big] \\
    &\quad + \sum_{x}\sum_{y\neq x} \big[\Delta^{y}u(x) - \Delta^{y}\tilde{u}(x)\big]
    \Big[\mu(x)[\Delta^{y}u(x)]_{-} - \tilde{\mu}(x)[\Delta^{y}\tilde{u}(x)]_{-}\Big] \\
    &\quad + \sum_{x} \big[-f(x,\mu)+f(x,\tilde{\mu})\big] \big[\mu(x)-\tilde{\mu}(x)\big] \\
    &\quad + \sum_{x}\sum_{y\neq x} 
    \big[ - b(x,\mu)\Delta^{y}u(x) + b(x,\tilde{\mu})\Delta^{y}\tilde{u}(x) \big] \big[\mu(x)-\tilde{\mu}(x)\big] \\
    &\quad + \sum_{x}\sum_{y\neq x} \big[\Delta^{y}u(x) - \Delta^{y}\tilde{u}(x)\big] \big[\mu(x)b(x,\mu) - \tilde{\mu}(x)b(x,\tilde{\mu})\big]
\end{align*}
Let us consider the three terms $(A)$, $(B)$ and $(C)$ defined as follows:
\begin{align*}
    (A)
    :&= \sum_{x}\sum_{y\neq x}
    \Bigg(
    \Bigg[\frac{1}{2}[\Delta^{y}u(x)]_{-}^{2} - \frac{1}{2}[\Delta^{y}\tilde{u}(x)]_{-}^{2}\Bigg]
    \big[\mu(x)-\tilde{\mu}(x)\big] \\
    &\qquad\qquad\qquad + \big[\Delta^{y}u(x) - \Delta^{y}\tilde{u}(x)\big]
    \Big[\mu(x)[\Delta^{y}u(x)]_{-} - \tilde{\mu}(x)[\Delta^{y}\tilde{u}(x)]_{-}\Big]
    \Bigg) \\
    &= \sum_{x}\sum_{y\neq x}
    \Bigg(
    \mu(x)\Bigg[
    \frac{1}{2} [\Delta^{y}u(x)]_{-}^{2} - \frac{1}{2} [\Delta^{y}\tilde{u}(x)]_{-}^{2}
    + [\Delta^{y}u(x)]_{-} [\Delta^{y}u(x) - \Delta^{y}\tilde{u}(x)]
    \Bigg] \\
    &\qquad\qquad\qquad
    - \tilde{\mu}(x)\Bigg[
    \frac{1}{2} [\Delta^{y}u(x)]_{-}^{2} - \frac{1}{2} [\Delta^{y}\tilde{u}(x)]_{-}^{2}
    + [\Delta^{y}\tilde{u}(x)]_{-} [\Delta^{y}u(x) - \Delta^{y}\tilde{u}(x)]
    \Bigg]
    \Bigg) ,\\
    (B)
    :&= \sum_{x} \big[-f(x,\mu)+f(x,\tilde{\mu})\big] \big[\mu(x)-\tilde{\mu}(x)\big] \\
    &= - \sum_{x} \big[f(x,\mu)-f(x,\tilde{\mu})\big] \big[\mu(x)-\tilde{\mu}(x)\big], \\
    %
    %
    (C)
    :&= \sum_{x}\sum_{y\neq x}
    \big[ - b(x,\mu)\Delta^{y}u(x) + b(x,\tilde{\mu})\Delta^{y}\tilde{u}(x) \big] \big[\mu(x)-\tilde{\mu}(x)\big] \\
    &\quad + \sum_{x}\sum_{y\neq x} \big[\Delta^{y}u(x) - \Delta^{y}\tilde{u}(x)\big] \big[\mu(x)b(x,\mu) - \tilde{\mu}(x)b(x,\tilde{\mu})\big] \\
    &= \sum_{x}\sum_{y\neq x}
    \mu(x)
    \Big[
    - b(x,\mu)\Delta^{y}u(x) + b(x,\tilde{\mu})\Delta^{y}\tilde{u}(x)  + b(x,\mu)\Delta^{y}u(x) - b(x,\mu)\Delta^{y}\tilde{u}(x)
    \Big] \\
    &\quad - \sum_{x}\sum_{y\neq x}
    \tilde{\mu}(x)
    \Big[
    - b(x,\mu)\Delta^{y}u(x) + b(x,\tilde{\mu})\Delta^{y}\tilde{u}(x)  + b(x,\tilde{\mu})\Delta^{y}u(x) - b(x,\tilde{\mu})\Delta^{y}\tilde{u}(x) 
    \Big] \\
    &= \sum_{x}\sum_{y\neq x} \big[ b(x,\mu) - b(x,\tilde{\mu}) \big] \big[ \Delta^{y}u(x)\tilde{\mu}(x) - \Delta^{y}\tilde{u}(x)\mu(x) \big].
\end{align*}
We need to show that 
\begin{equation*}
    \varphi'(t) = (A) + (B) + (C) \le 0.
\end{equation*}
Since $(C) \le |(C)|$, it will be easier to show
\begin{equation*}
    (A) + (B) + |(C)| \le 0.
\end{equation*}
Then,
\begin{align*}
    (A)
    &\leq \sum_{x}\sum_{y\neq x} 
    \Bigg(
    - \frac{1}{2} \mu(x) \Big[ [\Delta^{y}u(x)]_{-} - [\Delta^{y}\tilde{u}(x)]_{-} \Big]^2 
    - \frac{1}{2} \tilde\mu(x) \Big[ [\Delta^{y}u(x)]_{-} - [\Delta^{y}\tilde{u}(x)]_{-} \Big]^2
    \Bigg), 
    \intertext{\text{as $\frac12 A_-^2 - \frac12 B_-^2 + A_-(A-B) \leq -\frac12(A_- - B_-)^2$, and $\frac12 A_-^2 - \frac12 B_-^2 + B_-(A-B) \leq \frac12(A_- - B_-)^2$};}
    (B)
    &\leq - C_f \norm{\mu-\tilde{\mu}}_1^2,
    \intertext{\text{by Assumption (A1); finally, by assumption (A3)}}
    |(C)|
    &\leq
    \sum_{x}\sum_{y\neq x} L_b \left[ \prod_{z \neq x}\mu(z) + \prod_{z \neq x}\tilde{\mu}(z) \right] \norm{\mu-\tilde{\mu}}_1 \Big| \Delta^{y}u(x)\tilde{\mu}(x) - \Delta^{y}\tilde{u}(x)\mu(x) \Big|, \\
    &\leq \sum_{x}\sum_{y\neq x} L_b \left[ \prod_{z \neq x}\mu(z) + \prod_{z \neq x}\tilde{\mu}(z) \right] \norm{\mu-\tilde{\mu}}_1 \Big[ |\Delta^y u(x)||\mu(x)-\tilde{\mu}(x)| + \mu(x)|\Delta^y u(x) - \Delta^y \tilde{u}(x)| \Big] \\
    &\leq \sum_{x}\sum_{y\neq x} L_b \left[ \prod_{z \neq x}\mu(z) + \prod_{z \neq x}\tilde{\mu}(z) \right] \norm{\mu-\tilde{\mu}}_1 \mu(x)\big| [\Delta^y u(x)]_- - [\Delta^y \tilde{u}(x)]_- \big| \\
    &\qquad + \sum_{x}\sum_{y\neq x} L_b \left[ \prod_{z \neq x}\mu(z) + \prod_{z \neq x}\tilde{\mu}(z) \right] \norm{\mu-\tilde{\mu}}_1 \mu(x)\big| [\Delta^y u(x)]_+ - [\Delta^y \tilde{u}(x)]_+ \big| \\
    &\qquad + \sum_{x} L_b \left[ \prod_{z \neq x}\mu(z) + \prod_{z \neq x}\tilde{\mu}(z) \right] \norm{\mu-\tilde{\mu}}_1 \big| \mu(x) - \tilde{\mu}(x) \big| \sum_{y\neq x} \big|\Delta^y u(x)\big| \\
    &=: (C1) + (C2) + (C3).
\end{align*}
We estimate (C1) by doing
\begin{align*}
    \prod_{z \neq x}\mu(z) + \prod_{z \neq x}\tilde{\mu}(z) \leq 2;
\end{align*}
we estimate (C2) by keeping in mind that
\begin{align*}
    [\Delta^y u(x)]_+=[-\Delta^y u(x)]_ -= [\Delta^x u(y)]_-,
    \qquad
    [\Delta^y \tilde u(x)]_+=[\Delta^x \tilde u(y)]_-,
\end{align*}
and by relabelling the variables, so that
\begin{align*}
    (C2)
    &= \sum_{x}\sum_{y\neq x} L_b \left[ \prod_{z \neq x}\mu(z) + \prod_{z \neq x}\tilde{\mu}(z) \right] \norm{\mu-\tilde{\mu}}_1 \mu(x)\big| [\Delta^x u(y)]_- - [\Delta^x \tilde{u}(y)]_- \big| \\
    &= \sum_{y}\sum_{x\neq y} L_b \left[ \prod_{z \neq y}\mu(z) + \prod_{z \neq y}\tilde{\mu}(z) \right] \norm{\mu-\tilde{\mu}}_1 \mu(y)\big| [\Delta^y u(x)]_- - [\Delta^y \tilde{u}(x)]_- \big| \\
    &\leq \sum_{x}\sum_{y\neq x} L_b \bigg[ \mu(x) + \tilde{\mu}(x) \bigg] \big| [\Delta^y u(x)]_- - [\Delta^y \tilde{u}(x)]_- \big| \norm{\mu-\tilde{\mu}}_1.
\end{align*}
finally, we estimate (C3) with
\begin{align*}
    2 (d-1) L_b \norm{\Delta u}_{\infty} \norm{\mu - \tilde{\mu}}_1^2
\end{align*}
Combining all the terms, we end up with
\begin{align*}
    \frac{d}{dt}\varphi(t) 
    &\leq - C_f \norm{\mu - \tilde{\mu}}_1^2\\
    &\qquad + \sum_{x}\sum_{y\neq x} 
    \Bigg(
    - \frac12 \mu(x)^2 \Big| [\Delta^{y}u(x)]_- - [\Delta^{y}\tilde{u}(x)]_- \Big|^2 + 3 L_b \mu(x) \Big| [\Delta^{y}u(x)]_- - [\Delta^{y}\tilde{u}(x)]_- \Big| \norm{\mu - \tilde{\mu}}_1
    \Bigg) \\
    &\qquad + \sum_{x}\sum_{y\neq x} 
    \Bigg(
    - \frac12 \tilde{\mu}(x)^2 \Big| [\Delta^{y}u(x)]_- - [\Delta^{y}\tilde{u}(x)]_- \Big|^2 + L_b \tilde{\mu}(x) \Big| [\Delta^{y}u(x)]_- - [\Delta^{y}\tilde{u}(x)]_- \Big| \norm{\mu - \tilde{\mu}}_1 
    \Bigg) \\
    &\qquad + 2(d-1) L_b \norm{\Delta u}_{\infty} \norm{\mu-\tilde{\mu}}_1^2 \\
    &\leq \sum_{x}\sum_{y\neq x} 
    \Bigg(
    - \frac12 \mu(x)^2 \Big| [\Delta^{y}u(x)]_- - [\Delta^{y}\tilde{u}(x)]_- \Big|^2
    + 3 L_b \mu(x) \Big| [\Delta^{y}u(x)]_- - [\Delta^{y}\tilde{u}(x)]_- \Big| \norm{\mu - \tilde{\mu}}_1 \\
    &\qquad\qquad\qquad\qquad\qquad\qquad\qquad
    - \frac{C_f - 2(d-1)L_b \norm{\Delta u}_{\infty}}{2d(d-1)} \norm{\mu - \tilde{\mu}}_1^2 
    \Bigg) \\
    &\qquad + \sum_{x}\sum_{y\neq x} 
    \Bigg(
    - \frac12 \tilde{\mu}(x)^2 \Big| [\Delta^{y}u(x)]_- - [\Delta^{y}\tilde{u}(x)]_- \Big|^2
    +  L_b \tilde{\mu}(x) \Big| [\Delta^{y}u(x)]_- - [\Delta^{y}\tilde{u}(x)]_- \Big| \norm{\mu - \tilde{\mu}}_1 \\
    &\qquad\qquad\qquad\qquad\qquad\qquad\qquad
    - \frac{C_f - 2(d-1)L_b \norm{\Delta u}_{\infty}}{2d(d-1)} \norm{\mu - \tilde{\mu}}_1^2 
    \Bigg).
\end{align*}
We need to study when the polynomials
\begin{align*}
    &P(A,B) := -\frac{1}{2}A^2 + 3 L_b AB - \frac{C_f - 2(d-1)L_b \norm{\Delta u}_{\infty}}{2d(d-1)} B^2 \\
    &\tilde{P}(\tilde{A},B) := -\frac{1}{2}\tilde{A}^2 + L_b \tilde{A}B - \frac{C_f - 2(d-1)L_b \norm{\Delta u}_{\infty}}{2d(d-1)} B^2
\end{align*}
are $\leq 0$, where
\begin{align*}
    &A(x,y) := \mu(x) \Big| [\Delta^{y}u(x)]_- - [\Delta^{y}\tilde{u}(x)]_- \Big| \\
    &\tilde{A}(x,y) := \tilde\mu(x) \Big| [\Delta^{y}u(x)]_- - [\Delta^{y}\tilde{u}(x)]_- \Big| \\
    &B(x,y) := \norm{\mu - \tilde{\mu}}_1.
\end{align*}
It turns out that the following two conditions have to hold:
\begin{align*}
    &C_f > (d-1) [9 d L_b^2 + 2 \norm{\Delta u}_\infty L_b] \\
    &C_f > (d-1) [d L_b^2 + 2 \norm{\Delta u}_\infty L_b],
\end{align*}
the latter being implied by the former.

\emph{Step 2:}
Since $\varphi(T) \geq 0$ whenever $g$ is monotone, and $\varphi(0) = 0$, we have
\begin{align*}
    0 \leq \varphi(T) \leq \varphi(t) \leq \varphi(0) = 0 \quad \forall t \in [0,T],
\end{align*}
and therefore $\varphi$ is identically 0 on $[0,T]$ and so is $\frac{d}{dt}\varphi$. 
Let 
$$K = (d-1) [9 d L_b^2 + 2 \norm{\Delta u}_{\infty} L_b].$$ 
Fix now $t\in (0,T]$. Then
\begin{align*}
    0
    &= \frac{d}{dt}\varphi(t) \\
    &= (A) + (B) + (C) \\
    &\leq (A) - C_f \norm{\mu(t)-\tilde{\mu}(t)}_1^2 + (C) \\
    &= - (C_f - K) \norm{\mu(t)-\tilde{\mu}(t)}_1^2 - K \norm{\mu(t)-\tilde{\mu}(t)}_1^2 + (A) + (C) .
\end{align*}
By the proof in Step 1, we have
\[
- K \norm{\mu(t)-\tilde{\mu}(t)}_1^2 + (A) + (C) \leq 0
\] 
and thus we obtain 
\[
 (C_f - K) \norm{\mu(t)-\tilde{\mu}(t)}_1^2 \leq 0.
\] 
Therefore the assumption $C_f>K$ gives $\mu(t)= \tilde{\mu}(t)$ for any $t\in [0,T]$, as $\mu$ and $\tilde{\mu}$ are continuous. It follows that $u=\tilde{u}$ as the HJB equation admits a unique solution.
\end{proof}

\subsection{Gradient Estimates}\label{sec:gradient estimate}

We first provide an estimate depending of the time horizon $T$. Let as above 
\[
\norm{\Delta u}_\infty = \sup_{t \in [0,T]} \max_{x,y \in \Sigma} |\Delta^y u(t,x)|
\]

\begin{lemma}
   We have 
    \begin{equation}
        \norm{\Delta u}_{\infty}
        \le 2 \big( T \norm{f}_{\infty} + \norm{g}_{\infty} \big).
    \end{equation}
\end{lemma}
\begin{proof}
    Each component of the vector $\Delta u(t,x)$ can be estimated as
    \begin{align*}
        |\Delta^y u(t,x)|
        &= |u(t,y) - u(t,x)| \\
        &\le |u(t,y)| + |u(t,x)| \\
        &\le 2 \max\{|u(t,y)|, |u(t,x)|\} \\
        &\le 2 \max_{x \in \Sigma} |u(t,x)|.
    \end{align*}
    From the definition of the value function, and choosing the control $\alpha = 0$, we have
    \begin{align*}
        u(t,x)
        &= \inf_{\alpha} \E\left[\int_{t}^{T}\left(\frac{1}{2} \sum_{y \neq x} |\alpha_y(s,X(s))|^2 + f(X(s),\mu(s))\right) ds + g(X(T),\mu(T))\right] \\
        &\le \E\left[\int_{t}^{T} f(X(s),\mu(s))ds + g(X(T),\mu(T))\right] \\
        &\le (T-t)\norm{f}_{\infty} + \norm{g}_{\infty} \\
        &\le T \norm{f}_{\infty} + \norm{g}_{\infty}.
    \end{align*}
    The other bound follows as $\tfrac{1}{2} \sum_{y \neq x} |\alpha_y(s,X(s))|^2\geq 0$, thus 
    $\norm{u}_{\infty} \le  \big( T \norm{f}_{\infty} + \norm{g}_{\infty} \big)$, and therefore the claim follows. 
\end{proof}

Note that condition (A4), when employing the above bound, would imply a smallness condition on $T$, which we prefer to avoid. 
An upper bound for $\norm{\Delta u}_\infty$ that does not depend on $T$ can also be obtained, under an additional assumption. We first need a contraction estimate uniform in time.

\begin{lemma}[Synchronous coupling and exponential contraction]
\label{lem:coupling}
Let $\Sigma$ be a finite state space endowed with the discrete metric
\[
\rho(x,y) := \mathbbm 1_{\{x\neq y\}}.
\]
Fix a measurable flow of measures $(\mu(t))_{t\in[0,T]}$ and a feedback control
$\alpha^*(t,x)$.
Assume that there exists $\kappa>0$ such that
\[
b(x,\mu) \ge \kappa,
\qquad \forall x\in\Sigma,\ \mu\in\mathcal P(\Sigma).
\]

Let $X^{t,x}$ and $X^{t,y}$ be two controlled processes starting respectively from
$x$ and $y$ at time $t$, both evolving under the same feedback $\alpha^*$, with jump
intensities
\[
\lambda_s(x,z) := \alpha^{*,z}(s,x) + b(x,\mu(s)), \qquad z\neq x.
\]
Then there exists a coupling $(X_s,Y_s)_{s\ge t}$ of
$(X^{t,x},X^{t,y})$ such that
\[
\mathbb P(X_s \neq Y_s)
\le e^{- \kappa d (s-t)} \, \rho(x,y),
\qquad \forall s\ge t.
\]
In particular,
\[
\mathbb E[\rho(X_s,Y_s)] \le e^{- \kappa d (s-t)} \rho(x,y).
\]
\end{lemma}
\begin{proof}
For each $s\ge t$, the marginal dynamics of $X^{t,x}$ is governed by the generator
\[
(\mathcal L_s \varphi)(x)
= \sum_{z\neq x} \lambda_s(x,z)\big(\varphi(z)-\varphi(x)\big),
\qquad
\lambda_s(x,z) = \alpha^{*,z}(s,x) + b(x,\mu(s)).
\]
The same generator applies to $X^{t,y}$.

By assumption, $\lambda_s(x,z)\ge \kappa$ for all $x,z$ and $s$.
We decompose the jump rates as
\[
\lambda_s(x,z) = \kappa + \tilde\lambda_s(x,z),
\qquad \tilde\lambda_s(x,z) \ge 0.
\]

We construct a coupled process $(X_s,Y_s)$ as follows.
When $X_s \neq Y_s$,
\begin{itemize}
\item for each $z\in\Sigma$, with intensity $\kappa$ both components jump
simultaneously to $z$;
\item the remaining jump intensities $\tilde\lambda_s(x,z)$ and
$\tilde\lambda_s(y,z)$ are realised independently for $X$ and $Y$.
\end{itemize}
When $X_s=Y_s$, both components evolve synchronously with generator $\mathcal L_s$.
By construction, each marginal has generator $\mathcal L_s$.

Let $D_s := \rho(X_s,Y_s)=\mathbbm 1_{\{X_s\neq Y_s\}}$.
If $D_s=1$, a simultaneous jump occurs with total intensity $\kappa d$ and
forces $D$ to jump to zero. Hence, in the sense of generators,
\[
\frac{d}{ds}\mathbb E[D_s]
\le - \kappa d \, \mathbb E[D_s].
\]
Applying Gronwall's lemma yields
\[
\mathbb E[D_s] \le e^{-\kappa d(s-t)} D_t
= e^{-\kappa d(s-t)} \rho(x,y),
\]
which proves the claim.
\end{proof}

\begin{corollary}[Uniform Lipschitz bound for the value function]
Assume that $f(\cdot,\mu)$ is $L_f$-Lipschitz and $g(\cdot,\mu)$ is $L_g$-Lipschitz with respect to the discrete metric
$\rho$. Then, for all $t\in[0,T]$ and all $x,y\in\Sigma$,
\[
|u(t,x)-u(t,y)| \le \max\Big\{ L_f \frac{1}{\kappa d}, L_g\Big\}  \rho(x,y),
\qquad
\|\Delta u\|_\infty \le \max\Big\{ L_f \frac{1}{\kappa d}, L_g\Big\} .
\]
\end{corollary}

Note that assumption (A4), when applying this bound, implies a smallness condition on $L_f, L_g, L_b$. 

\begin{proof}
Let $\alpha^*(t,x)$ be an optimal feedback for the control problem starting in $(t,x)$ with fixed flow $\mu(t)$, and let $X$ and $Y$ be as in the above lemma. 
We consider now the difference of the value functions: 
\begin{align*}
    u(t,x) - u(t,y) 
    &\leq J(t,x,\alpha^*) - J(t,y,\alpha^*) \\
    &= \E\bigg[ \int_{t}^{T} [f(X_s,\mu_s) - f(Y_s,\mu_s)]ds  
    + g(X_T,\mu_T) - g(Y_T,\mu_T)\bigg]\\
    &\leq L_f \int_{t}^{T} \E\left[ \rho(X_s , Y_s ) \right] ds
    + L_g \E\left[ \rho(X_T , Y_T ) \right]\\
    &\leq L_f \int_t^T e^{-\kappa d  (s-t)} \rho(x,y)ds
    + L_g e^{-\kappa d  (T-t)} \rho(x,y)  \\
    &= \Big( L_f \frac{1}{\kappa d} \big(1- e^{-\kappa d (T-t)} \big) 
    + L_g e^{-\kappa d (T-t)} \Big)\rho(x,y)  \\
    &\leq \max\Big\{ L_f \frac{1}{\kappa d}, L_g\Big\} \rho(x,y).
\end{align*}
The opposite inequality is derived in the same way, and thus the claim follows. 
\end{proof}

\section{Conclusion}\label{sec:conclusion}

We study the uniqueness of finite-state mean field game systems with non-separable Hamiltonians induced by distribution-dependent transition rates. In this setting, classical Lasry-Lions arguments do not apply directly. We prove that uniqueness can still be obtained under a combination of monotonicity assumptions on the costs and quantitative conditions on the interaction term.
The results provide a tractable framework for MFG models with interaction at the level of transition rates and highlight how classical techniques can be adapted beyond the separable case.

Possible extensions include relaxing the assumptions, considering more general interaction structures, and studying the associated master equation.


\end{document}